\documentclass[12pt]{amsart}
\usepackage{amsthm}
\usepackage{amsmath,amssymb,euscript,mathrsfs,bm}
\usepackage{eepic,bbm}
\usepackage{bbold}
\usepackage{dsfont}
\usepackage{ebgaramond}
\usepackage{setspace}
\usepackage{tikz-cd}
\usepackage[colorlinks=true,linkcolor=black,citecolor=black,urlcolor=black]{hyperref}
\usepackage{cleveref}

\newcommand{\define}{\textit}

\newcommand{\C}{\mathds{C}}
\newcommand{\N}{\mathds{N}}
\newcommand{\Q}{\mathds{Q}}
\newcommand{\Z}{\mathds{Z}}
\renewcommand{\P}{\mathds{P}}

\newcommand{\id}{\mathrm{id}}
\newcommand{\pt}{\mathrm{pt}}

\DeclareMathOperator{\Hom}{Hom}

\newcommand{\Fl}{ {Fl} }
\newcommand{\ee}{\mathrm{e}}    
\newcommand{\dlim}{\varinjlim}
\newcommand{\ilim}{\varprojlim}

\newtheoremstyle{scthm}%
{}{}{\itshape}{}{\scshape}{.}{ }{}

\newtheoremstyle{scdef}%
{}{}{}{}{\scshape}{.}{ }{}

\theoremstyle{scthm}

\newtheorem{theorem}{Theorem}
\newtheorem{lemma}[theorem]{Lemma}
\newtheorem{proposition}[theorem]{Proposition}

\newtheorem{theoremAlph}{Theorem}

\newtheorem*{thm*}{Theorem}
\newtheorem*{cor*}{Corollary}
\newtheorem*{lem*}{Lemma}

\begin{document}

\newcommand{\isom}{\cong}
\renewcommand{\setminus}{\smallsetminus}
\renewcommand{\phi}{\varphi}
\renewcommand{\hat}{\widehat}

\title[\v{C}ech cohomology of infinite projective spaces]{\v{C}ech cohomology of infinite projective spaces,\\flag manifolds, and related spaces}
\author{David Anderson}
\address{Department of Mathematics, The Ohio State University, Columbus, OH 43210}
\email{anderson.2804@math.osu.edu}

\author{Matthias Franz}
\address{Department of Mathematics, University of Western Ontario, London, Ont. N6A 5B7, Canada}
\email{mfranz@uwo.ca}

\thanks{DA was partially supported by NSF CAREER DMS-1945212, and by a Membership at the Institue for Advanced Study funded by the Charles Simonyi Endowment.  MF was supported by an NSERC Discovery Grant.}

\date{December 15, 2025}


\begin{abstract}
  We compute the \v{C}ech cohomology ring of a countable product of infinite projective spaces, and that of an infinite flag manifold.  The method of our first result in fact computes the cohomology ring of a countably infinite product of paracompact Hausdorff spaces, under some mild assumptions.
  \par\noindent
\end{abstract}

\maketitle

\section{Introduction}

Let $\C\P^\infty$ be infinite-dimensional complex projective space, topologized as a rising union of finite-dimensional projective spaces, $\C\P^\infty=\bigcup \C\P^k$, or equivalently, as the quotient space $\C\P^\infty = (\C^\infty\setminus 0)/\C^*$, where $\C^\infty$ has again the colimit topology of the spaces~$\C^{k}$.  The latter description makes $\C\P^\infty$ a model for the classifying space of $\C^*$, and we have the well known calculation
\[
  H^*\C\P^\infty = H^*B\C^* = H^*_{\C^*}(\pt) = \Z[t],
\]
a polynomial ring in one generator of degree $2$, often taken to be the Chern class of the tautological line bundle.  Here cohomology is taken with integer coefficients, and one may use any standard cohomology theory: the space $\C\P^{\infty}$ is locally contractible, paracompact, and Hausdorff, so singular, \v{C}ech, Alexander--Spanier, and sheaf cohomology all agree.  (See, e.g., \cite[Corollary~6.9.5, Exercise~6.D3]{spanier} and \cite[\S I.7 and Chapter~III]{bredon}.)

Likewise, the cohomology ring of a finite product $(\C\P^\infty)^{\times n}$ is well known to be a polynomial ring in $n$ generators of degree $2$:
\[
  H^*(\C\P^\infty)^{\times n} = H^*_{(\C^*)^n}(\pt) = \Z[t_1,\ldots,t_n].
\]
Again, this calculation is valid in any standard cohomology theory.

The first aim of this note is to compute the analogous ring for the countably infinite product of infinite projective spaces.  The principal bundle $(\C^\infty\setminus0)^{\times \N} \to (\C\P^\infty)^{\times \N}$ has contractible total space, so $(\C\P^\infty)^{\times \N}$ is a model for the classifying space of the infinite-dimensional torus $T = (\C^*)^{\times \N}$,
with the caveat that this bundle is not locally trivial in the product topology.  (It becomes locally trivial in the finer box topology, but we will not use this.)  
Equivariant cohomology with respect to~$T$ comes up in several contexts---see, e.g., \cite{liou-schwarz,lls,infschub}---and so does the cohomology of its classifying space.

Using \v{C}ech cohomology\footnote{Throughout, we use the fact that the Alexander--Spanier and \v{C}ech theories agree---see \cite[Exercise~6.D3]{spanier} or \cite[\S I.7]{bredon}---so for brevity we often refer simply to \v{C}ech cohomology.
} with integer coefficients, we will show that
\begin{equation}\label{e.cor1}
  H^* (\C\P^\infty)^{\times \N} = H^*_T(\pt) = \Z[t_1,t_2,\ldots],
\end{equation}
a polynomial ring in countably many variables of degree $2$, where $t_i$ may be regarded as the Chern class of the tautological line bundle pulled back along the projection to the $i$th factor.  In fact, the same holds with $\Q$ coefficients: see the Corollary at the end of this introduction.  The calculation uses a construction of $(\C\P^\infty)^{\times \N}$ as an inverse limit of spaces $(\C\P^\infty)^{\times n}$.

This time, the choice of cohomology theory matters.  The rational version of the asserted isomorphism \eqref{e.cor1} says that $H^2\big((\C\P^\infty)^{\times \N}; \Q\big) = \bigoplus_{i\geq 1} \Q t_i$ is a vector space of countably infinite rank, so it is not the dual of any vector space (e.g., by the Erd\H{o}s--Kaplansky theorem). The universal coefficient theorem therefore implies that it cannot be the singular cohomology of any space.

Nevertheless, the calculation of $H^*(\C\P^\infty)^{\times \N}$ gives the ``expected'' answer, in the following sense.
Let $X_{\leq n} = (\C\P^\infty)^{\times n}$, and for $n\geq n'$, consider the projection $X_{\leq n} \to X_{\leq n'}$ onto the first $n'$ factors.
So the $X_{\leq n}$ form an inverse system, with limit
\[
  X = \ilim_n X_{\leq n} = (\C\P^\infty)^{\times \N}.
\]
Reversing the arrows, there is a natural homomorphism of cohomology rings
\begin{equation}\label{e.colim-map}
  \dlim H^*X_{\leq n} \to H^*(\ilim X_{\leq n}).
\end{equation}
If this map were known to be an isomorphism, we would obtain
\[
  H^*X = \dlim_n H^* X_{\leq n} = \dlim_n \Z[t_1,\ldots,t_n] = \Z[t_1,t_2,\ldots],
\]
as asserted by \eqref{e.cor1}.

When the spaces involved are compact, the continuity property for \v{C}ech--Alexan\-der--Spanier cohomology (reviewed below) implies that the homomorphism \eqref{e.colim-map} is an isomorphism.
For inverse limits of non-compact spaces, however, simple counterexamples show that \eqref{e.colim-map} can fail to be an isomorphism in any reasonable cohomology theory.\footnote{%
  For instance, take $X_n$ to be the cylinder $S^1\times [n,\infty)$, with $X_n \to X_{n'}$ the inclusion.
  Each inclusion induces isomorphisms on cohomology, so $\dlim H^1 X_n = \Z$.  But the inverse limit is
  the intersection $\bigcap X_n = \emptyset$, so $H^1(\ilim X_n) =0$.  See also \cite{watanabe}.}
In our setting, the spaces~$X_{\leq n} = (\C\P^\infty)^{\times n}$ are not compact, so further argument is required.

The reasoning sketched above has led several authors (including the present first author) to assert the calculation \eqref{e.cor1} without sufficient justification, so it seems worth providing details.  In particular, this note fills a gap left in \cite[\S A.8]{ecag}.

We will establish the isomorphism~\eqref{e.cor1} in a more general context.  
From now on, unless otherwise indicated, we use \v{C}ech cohomology with coefficients in a countable principal ideal domain $R$.  (Typical examples are $\Z$, $\Q$, finite fields, and subrings of $\Q$.) 
Recall that a CW~complex is said to be of \define{finite type} if it has only finitely many cells in each dimension.

\begin{theoremAlph}
  \label{thm:A}
  Let $X_1$,~$X_{2}$,~\dots  be a countable family of CW~complexes of finite type.
  Write $X$ for their product, and set $X_{\le n}=X_{1}\times\dots\times X_{n}$.
  The natural homomorphism
  \[
    \dlim_n H^*(X_{\leq n}) \xrightarrow{\sim} H^*X,
  \]
  induced by the canonical projections~$X\to X_{\le n}$, is an isomorphism of graded $R$-algebras.
\end{theoremAlph}

The space $X_{\leq n}$ is again a CW~complex since each~$X_{k}$ has at most countably many cells \cite[Theorem~A.6]{hatcher}.
And since CW~complexes are locally contractible and paracompact Hausdorff---see, e.g., \cite{LundellWeingram:1969} or \cite[Appendix]{hatcher}---\v{C}ech cohomology agrees with singular cohomology for these spaces.
(This fact also shows that the countable product~$X$ is not a CW~complex in general.)

The calculation of the cohomology of $(\C\P^\infty)^{\times \N}$ is an immediate consequence:

\begin{cor*}
  For \v{C}ech cohomology with coefficients in any countable principal ideal domain $R$, we have
  \[
   H^*\Bigl( (\C\P^\infty)^{\times \N} \Bigl) \isom \dlim_n R[t_1,t_2,\ldots,t_n] = R[t_1,t_2,\ldots].
  \]
  In particular, this holds with $R=\Z$ or $R=\Q$.
\end{cor*}

Using this calculation together with an explicit homotopy equivalence, we will deduce a computation of the cohomology ring of the infinite flag manifold (\Cref{thm:B} in \Cref{s.flags}).

\section{Proof of \Cref{thm:A}}

We first collect the required ingredients.  Recall that $R$ indicates a countable PID, and coefficients for cohomology are assumed to be taken in $R$, when unspecified.  

We use the general {\it tautness} property of \v{C}ech cohomology:
this says that for any closed subspace $Y$ of a paracompact space $Z$ {and coefficients in any $R$-module $M$}, the natural homomorphism
\[
  \dlim H^k(U;{M}) \to H^k(Y;{M})
\]
is an isomorphism, where the direct limit is over all neighborhoods $U$ of $Y$ in $Z$ (see \cite[Theorem~6.6.2]{spanier}).

There is also the above-mentioned {\it continuity} property: for an inverse system of compact Hausdorff spaces $Z_n$, with limit $Z$, the homomorphism
\[
 \dlim H^k(Z_n;{M}) \to H^k(Z;{M})
\]
is an isomorphism, again for any $R$-module $M$ \cite[Theorem~6.6.6]{spanier}.

We make use of the following general K\"unneth formula for \v{C}ech cohomology.
\begin{proposition}
  \label{thm:bartik}
  Assume that $P$ is locally contractible and that $P\times Y$ is paracompact Hausdorff.
  Then for any~$k\ge0$ there is a natural isomorphism
\[
  H^k( P \times Y) \isom \bigoplus_{i=0}^k H^i(P; H^{k-i}(Y)), \qedhere
\]
functorial in $P$ and $Y$.
\end{proposition}
When $P$ is compact, a proof via the spectral sequence of a double complex may be found in Bredon's book. (Use the short exact sequence of \cite[Theorem~IV.7.6]{bredon} together with the universal coefficient theorem for sheaf cohomology \cite[Theorem~II.15.3]{bredon}.)  
The version stated above is due to Bartik, and uses a natural isomorphism between $k$th \v{C}ech cohomology and homotopy classes of maps into the Eilenberg--MacLane space $K({R},k)$ \cite[Theorem~C(b)]{bartik}.
%

In preparation for the proof of our main result, we state some lemmas.
Writing $X_n^{(k)}$ for the $(k+2)$-skeleton of~$X_{n}$, we have an isomorphism
\begin{equation}
  \label{eq:restriction-Xk-Xkp}
  H^i(X_n;M) \to H^i(X_n^{(k)};M)
\end{equation}
for~$i\le k+1$ and~$M=R$, hence also for~$i\le k$ and any~$M$ by the universal coefficient theorem for singular cohomology.

In addition to the spaces~$X$ and~$X_{\le n}$ we introduce
\begin{equation*}
  X_{>n} = \prod_{m>n} X_m,
\end{equation*}
and we define $X^{(k)}$, $X_{\leq n}^{(k)}$, and $X_{>n}^{(k)}$
as the analogous products involving the spaces $X_{m}^{(k)}$. Note that the latter three products are all compact. Moreover, $X_{\leq n}^{(k)}$ contains the $(k+2)$-skeleton of~$X_{\leq n}$, so in fact the $(k+2)$-skeleta of $X_{\leq n}^{(k)}$ and of $X_{\leq n}$ are the same, which yields an isomorphism of the form~\eqref{eq:restriction-Xk-Xkp}  for~$X_{\leq n}$ in place of~$X_{n}$.

CW~complexes are not only paracompact, but in fact \define{hereditarily paracompact}
in the sense that every subspace is again paracompact; see~\cite[Corollary~II.4.4]{LundellWeingram:1969}.
Because we can express the countable product~$X$ as the inverse limit
\begin{equation*}
  X = \ilim_{n} X_{\le n},
\end{equation*}
this implies that $X$ is again paracompact \cite[Lemma~7]{DydakGeoghegan:1986}, and so are all~$X_{>n}$.

\begin{lemma}
  \label{thm:iso-X-Xkn}
  The map
  \begin{equation*}
    H^{i}(X) \to H^{i}(X_{\le n}^{(k)}\times X_{>n})
  \end{equation*}
  is an isomorphism for any~$n\ge0$ and any~$i\le k$.
\end{lemma}

\begin{proof}
  By \Cref{thm:bartik}, this reduces to the fact that the map
  \begin{equation*}
    H^{*}\bigl(X_{\le n};H^{*}(X_{>n})\bigr) \to H^{*}\bigl(X_{\le n}^{(k)};H^{*}(X_{>n})\bigr)
  \end{equation*}
  is an isomorphism in total degree~$i\le k$, which is a special case of what has been discussed above.
\end{proof}

\begin{lemma}
  \label{thm:iso-dlim-V}
  For any~$n\ge0$ and any~$i\le k$ there is an isomorphism
  \begin{equation*}
    \dlim_{V} H^{i}(V\times X_{>n}) = H^{i}(X^{(k)}_{\le n}\times X_{>n})
  \end{equation*}
  where $V$ ranges over the open neighborhoods of~$X^{(k)}_{\le n}$ in~$X_{\le n}$.
\end{lemma}

\begin{proof}
  Note that $V\times X_{>n}$ is the inverse limit of the spaces~$V\times X_{n+1}\times\dots\times X_{m}$ with $m\ge n+1$.
  The latter are subspaces of CW~complexes and therefore hereditarily paracompact, which ensures that $V\times X_{>n}$
  itself is paracompact.

  We may therefore appeal to \Cref{thm:bartik}, which together with tautness gives isomorphisms
  \begin{align*}
    \dlim_{V} H^{i}(V\times X_{>n}) &= \dlim_{V} \bigoplus_{j} H^{j}\bigl(V;H^{i-j}(X_{>n})\bigr) \\
    &= \bigoplus_{j} H^{j}\bigl(X^{(k)}_{\le n};H^{i-j}(X_{>n})\bigr) = H^{i}(X^{(k)}_{\le n}\times X_{>n}).
    \qedhere
  \end{align*}
\end{proof}

\begin{lemma}
  \label{thm:iota}
  For any~$i\le k$ the inclusion~$\iota\colon X^{(k)}\hookrightarrow X$ induces an isomorphism
  \begin{equation*}
    \iota^{*}\colon H^{i}(X) = H^{i}(X^{(k)}).
  \end{equation*}
\end{lemma}

\begin{proof}
  Consider the open neighborhoods of $X^{(k)} \subseteq X$ that are of the form
  \begin{equation*}
    U = V \times X_{>n},
  \end{equation*}
  where $V$ is an open neighborhood of $X_{\leq n}^{(k)}$ in $X_{\leq n}$, for some $n$.
  The collection $\mathcal{U}$ of all such $U$ is cofinal in the directed set of all open neighborhoods of $X^{(k)}$,
  by compactness of~$X^{(k)}$ and the definition of the product topology.
  Let $\mathcal{U}_n\subseteq \mathcal{U}$ be the subset of all $U$ that can be written in the above form for a given $n$,
  so $\mathcal{U}_n \subset \mathcal{U}_{n'}$ for $n\leq n'$.
  In this notation, \Cref{thm:iso-dlim-V} reads
  \begin{equation*}
    \dlim_{U\in \mathcal{U}_n} H^i(U) = H^{i}(X^{(k)}_{\le n}\times X_{>n})
  \end{equation*}
  for all $i\leq k$.

  Using this identity together with \Cref{thm:iso-X-Xkn} and the tautness property for~$X^{(k)} \subseteq X$ gives
  \begin{align*}
    H^i(X^{(k)}) &\isom \dlim_{U\in \mathcal{U}} H^i(U) = \dlim_n \dlim_{U\in \mathcal{U}_n} H^i(U) \\
    &\isom \dlim_n H^i(X_{\leq n}^{(k)} \times X_{>n})
    = \dlim_n H^i(X) = H^{i}(X)
  \end{align*}
  for all $i\leq k$, which completes the proof.
\end{proof}

Now we will show that the natural homomorphism $\alpha\colon \dlim_n H^kX_{\leq n} \to H^kX$
is an isomorphism for any fixed degree $k$, proving Theorem~A.  
We have a commuting square
\begin{equation}\label{e.square}
\begin{tikzcd}
  \displaystyle{\dlim_n H^k(X_{\leq n})} \ar[d,"j", center yshift=-1.1ex, start anchor=center] \ar[r,"\alpha"] & H^kX \ar[d,"\iota^*"] \\
  \displaystyle{\dlim_n H^k(X_{\leq n}^{(k)})} \ar[r,"\beta"] & H^kX^{(k)}.
\end{tikzcd}
\end{equation}
Since $X^{(k)} = \ilim_n X_{\leq n}^{(k)}$ is a limit of compact spaces, the homomorphism $\beta$ is an isomorphism by the continuity property.
The homomorphism $j$ is a colimit of isomorphisms $H^k(X_{\leq n}) \to H^k(X_{\leq n}^{(k)})$, so it is an isomorphism itself.
Finally, the map~$\iota^{*}$ is an isomorphism by \Cref{thm:iota}, and we conclude that so is $\alpha$.
\qed

\section{Flag manifolds}\label{s.flags}

Infinite-dimensional flag manifolds are closely related to products of projective spaces.  Let $\Fl(1,\ldots,n;\C^\infty)$ be the space parametrizing chains of subspaces $V_1 \subset \cdots \subset V_n \subset \C^\infty$ with~$\dim V_{i}=i$, topologized as a union
\[
  \Fl(1,\ldots,n;\C^\infty) = \bigcup_m \Fl(1,\ldots,n;\C^m).
\]
A standard calculation shows that
\[
  H^*\Fl(1,\ldots,n;\C^\infty) = \Z[t_1,\ldots,t_n]
\]
is again a polynomial ring in $n$~variables of degree~$2$; see, e.g., \cite[\S2.5]{ecag}.
Since the space $\Fl(1,\ldots,n;\C^\infty)$ is locally contractible, paracompact, and Hausdorff, any usual cohomology theory will do: in particular, singular and \v{C}ech cohomology agree.

The {\it infinite flag manifold} $\Fl(\C^\infty)$ parametrizes chains of subspaces $V_1 \subset V_2 \subset \cdots \subset \C^\infty$, with $\dim V_i = i$, and is topologized as an inverse limit of partial flag manifolds:
\[
  \Fl(\C^\infty) = \ilim_n \Fl(1,\ldots,n;\C^\infty),
\]
where the maps in the inverse system are projections onto the first part of a chain.  As with $(\C\P^\infty)^{\times \N}$, this space is no longer locally contractible, so the choice of cohomology theory matters.

\begin{theoremAlph}
  \label{thm:B}
  Using \v{C}ech cohomology with coefficients in a countable principal ideal domain $R$, we have
  \[
    H^*\Fl(\C^\infty) = H^*(\C\P^\infty)^{\times\N} = R[t_1,t_2,\ldots],
  \]
  a polynomial ring in countably many variables of degree $2$.
\end{theoremAlph}

With $\Q$ coefficients, it seems feasible to prove this directly by following the pattern of the previous argument, using a version of the Leray--Hirsch theorem in place of the K\"unneth formula, but we have not pursued this line of reasoning.  Instead, we will deduce \Cref{thm:B} from the calculation of $H^*(\C\P^\infty)^{\times \N}$ by constructing a homotopy equivalence between $\Fl(\C^\infty)$ and $(\C\P^\infty)^{\times \N}$.  Here we will be using the fact that \v{C}ech cohomology satisfies the homotopy axiom \cite[Theorem~6.5.6]{spanier}.

In general, suppose $X_n$ and $Y_n$ are two inverse systems, with respective limits $X = \ilim X_n$ and $Y=\ilim Y_n$.
If $f_n \colon X_n \to Y_n$ and $g_n \colon Y_n \to X_n$ are homotopy inverse maps of the systems with homotopies respecting the inverse systems,
then the induced maps $f\colon X \to Y$ and $g\colon Y \to X$ are homotopy inverses.

To prove \Cref{thm:B}, it suffices to construct homotopy equivalences between the partial flag manifold $\Fl(1,\ldots,n;\C^\infty)$ and $(\C\P^\infty)^{\times n}$, compatible with the projections.

Let $H_n = \Hom(\C^n,\C^\infty)$ be the vector space of $\infty\times n$ matrices with finitely many non-zero entries. Define the subspaces $H^{\circ\circ}_n \subset H^\circ_n \subset H_n$ by
\[
H^{\circ\circ}_n = \{ \text{linear } f\colon \C^n \to \C^\infty \,|\, f(v) \neq 0 \text{ for all }v\neq 0\}
\]
and
\[
H^\circ_n = \{ \text{linear } f\colon \C^n \to \C^\infty \,|\, f(\ee_i) \neq 0 \text{ for }1\leq i\leq n\},
\]
where $(\ee_{i})$ denotes the standard basis for~$\C^{n}$.
So $H^{\circ\circ}_n$ is the space of full-rank matrices---that is, of linear embeddings of $\C^n$ in $\C^\infty$---and $H^\circ_n$ is the space of matrices with all columns nonzero.  These spaces form inverse systems, via maps $H_n \to H_{n'}$ for $n\geq n'$, by projecting onto the first $n'$ columns.

The group $B_n$ of upper-triangular matrices in $GL_n$ acts freely on $H^{\circ\circ}_n$ by right multiplication, with quotient
\[
  H^{\circ\circ}_n/B_n = \Fl(1,\ldots,n;\C^\infty).
\]
The diagonal torus $T_n = (\C^*)^n$ also acts, with quotient
\[
 H^{\circ\circ}_n/T_n = \Fl^{\rm split}(1,\ldots,n;\C^\infty),
\]
where $\Fl^{\rm split}(1,\ldots,n;\C^\infty)$ is the ``split'' flag manifold, parametrizing $n$ linearly independent lines $L_1$,~\dots,~$L_n$ in $\C^\infty$.
We equip it with the quotient topology coming from the above identity.  
There is a natural projection $\Fl^{\rm split}(1,\ldots,n;\C^\infty) \to \Fl(1,\ldots,n;\C^\infty)$, defined by $V_i = L_1\oplus \cdots \oplus L_i$.  Choose a Hermitian metric on $\C^\infty$---say, by writing $\C^\infty$ as the union of $\C^m$ according to the standard bases, and equipping each $\C^m$ with the standard Hermitian metric.  This defines a section $\sigma$ of the projection map, by splitting the chain $V_1 \subset \cdots \subset V_n$ orthogonally with respect to the metric.  Using Gram--Schmidt orthonormalization, this makes $\Fl(1,\ldots,n;\C^\infty)$ a deformation retract of $\Fl^{\rm split}(1,\ldots,n;\C^\infty)$.  (To ensure that the deformation is compatible with the inverse system, one uses Gram--Schmidt to {\it simultaneously} move all the lines $L_i$ into orthogonal position.)

Let $\Fl(\C^\infty)$ be the inverse limit of the spaces~$\Fl(1,\ldots,n;\C^\infty)$, and define $\Fl^{\rm split}(\C^\infty)$, $H$, $H^{\circ}$ and $H^{\circ\circ}$ analogously.
Note that $\Fl(\C^\infty)$ is the inverse limit of the system~$H^{\circ\circ}_{n}/B_{n}$, and $\Fl^{\rm split}(\C^\infty)$ that of the system~$H^{\circ\circ}_{n}/T_{n}$.

The sections $\sigma$ described above are compatible with the inverse system maps: for $n\geq n'$ the diagram
\[
\begin{tikzcd}
  \Fl^{\rm split}(1,\ldots,n;\C^\infty)
    \ar[r]
    \ar[d]
  &
  \Fl^{\rm split}(1,\ldots,n';\C^\infty)
    \ar[d]  \\
  \Fl(1,\ldots,n;\C^\infty)
    \ar[r]
    \ar[u, bend left=20, dashed, "\sigma"]
  &
  \Fl(1,\ldots,n';\C^\infty)
    \ar[u, bend right=20, dashed, swap, "\sigma"]
\end{tikzcd}
\]
commutes.  Taking inverse limits, it follows that the projection $\Fl^{\rm split}(\C^\infty) = H^{\circ\circ}/T \to \Fl(\C^\infty) = H^{\circ\circ}/B$ is a homotopy equivalence.

On the other hand, $T_n$ also acts on $H^\circ_n$, with quotient
\[
  H^\circ_n/T_n = (\C\P^\infty)^{\times n}.
\]
For the inverse limit we get $H^\circ/T=(\C\P^\infty)^{\times \N}$.
The projections making $H^{\circ\circ}_n$ and $H^\circ_n$ inverse systems induce the projections of flag manifolds and products of projective spaces.

Now it suffices to construct a homotopy equivalence between $\Fl^{\rm split}(1,\ldots,n;\C^\infty)$ and $(\C\P^\infty)^{\times n}$, compatible with projections.  The following lemma completes the proof of \Cref{thm:B}.

\begin{lem*}
  The inclusions $\alpha\colon H^{\circ\circ}_n \hookrightarrow H^\circ_n$ are $T_n$-equivariant homotopy equivalences, via maps which respect the inverse systems.
\end{lem*}

In fact, we will prove the lemma directly on the limit spaces $H^{\circ\circ}$ and $H^\circ$, in such a way that the constructions evidently restrict compatibly to $H^{\circ\circ}_n$ and $H^{\circ}_n$ for finite $n$.

\begin{proof}
  We construct a map $\beta\colon H^\circ \hookrightarrow H^{\circ\circ}$ as follows.  The idea is to send $f\colon \C^\infty \to \C^\infty$ (with $f(\ee_i)\neq 0$ for all~$i$) to an injective map $\hat{f}\colon\C^\infty \to \C^\infty\otimes \C^\infty$, and then compose with an isomorphism $\nu\colon \C^\infty\otimes \C^\infty \to \C^\infty$ to obtain $\beta(f)$.

  First, given any $f\in H = \Hom(\C^\infty,\C^\infty)$, let $\hat{f}\colon \C^\infty \to \C^\infty\otimes \C^\infty$ be defined by
  \[
    \hat{f}(v) = \sum_i v_i f(\ee_i)\otimes \ee_i.
  \]
Since each $f(\ee_i)$ is nonzero, the vectors $\{ f(\ee_i) \otimes \ee_i \}_{i=1,2,\ldots}$ are linearly independent.  So $\hat{f}$ is injective, as claimed.

  Next let $\nu\colon \N \times \N \to \N$ be the bijection defined by enumerating pairs in the order indicated by reading matrix coordinates below:
  \[
    \begin{array}{ccccc}
      1 & 3 & 6 & 10 & \cdots \\
      2 & 5 & 9 & \cdots \\
      4 & 8 & \cdots \\
      7 & \cdots \\
      \vdots
    \end{array}
  \]
  That is, the enumeration of $\N\times\N$ proceeds as
  \[
   (1,1), \, (2,1), \, (1,2), \, (3,1), \, (2,2), \, (1,3),\, (4,1), \, (3,2), \, (2,3), \, (1,4),\, \ldots
  \]
  Reusing notation, let $\nu\colon \C^\infty \otimes \C^\infty \xrightarrow{\sim} \C^\infty$ be the isomorphism defined by $\ee_i \otimes \ee_j \mapsto \ee_{\nu(i,j)}$.

  Given $f\in H^\circ$ and $v = (v_1,v_2,\ldots) \in \C^\infty$, define $\beta(f)$ by
  \[
    \beta(f)(v) = \nu\bigl(\hat{f}(v)\bigr) = \nu\big( v_1 f(\ee_1)\otimes \ee_1 + v_2 f(\ee_2) \otimes \ee_2 + \cdots \big).
  \]
  In terms of matrices, $\beta$ sends
  \begin{equation}\label{e.mtx}
  f= \left[ \begin{array}{cccc}
     f_{11} & f_{12} & f_{13} & \cdots \\
     f_{21} & f_{22} & f_{23} & \cdots \\
     f_{31} & f_{32} & f_{33} & \cdots \\
     \vdots & \vdots & \vdots & \ddots
    \end{array}
   \right]
   \qquad \text{to} \qquad
  \beta(f)= \left[ \begin{array}{cccc}
     f_{11} & 0 & 0 & \cdots \\
     f_{21} & 0 & 0 & \cdots \\
     0 & f_{12} & 0 & \cdots \\
     f_{31} & 0 & 0 & \cdots \\
     0 & f_{22} & 0 & \cdots \\
     0 &  0  & f_{13} & \cdots \\
     \vdots & \vdots & \vdots & \ddots
    \end{array}
   \right].
  \end{equation}
  Note that each column has only finitely many nonzero entries.
  Since no column of $f$ is zero, the columns of $\beta(f)$ are evidently linearly independent.

  Both $\alpha$ and $\beta$ are torus-equivariant, since they preserve the column-indices of entries.  And they are evidently compatible with the inverse system---in fact, the above description applies verbatim to maps $\beta\colon H_n^\circ \hookrightarrow H_n^{\circ\circ}$, defined analogously for all $n$.

  Finally, we construct homotopies $\alpha\circ \beta \simeq \id_{H^\circ}$ and $\beta\circ \alpha \simeq \id_{H^{\circ\circ}}$.  Since $\alpha$ is just a natural inclusion, we suppress it from the notation.  (So both $\alpha\circ \beta$ and $\beta\circ \alpha$ are given by the same formula that defines $\beta$; in matrices this is \eqref{e.mtx}.)  We use the straight line homotopy
  \[
    (f,t) \mapsto t f + (1-t) \beta(f).
  \]

  When $f \in H^\circ$, so each column is nonzero, it is easy to see the same is true of $t f + (1-t) \beta(f)$ for each $t\in [0,1]$.  (For $t=1$, this is true by assumption; for $t=0$ it holds because the entries of the columns of $\beta(f)$ are the same as those of $f$, but with $0$'s inserted.  For $0<t<1$, the first nonzero entry in column $j$ of $f$ appears strictly above the corresponding entry of $\beta(f)$, so it remains nonzero in $t f + (1-t)\beta(f)$---except if $j=1$ and the first nonzero entry is $f_{11}$ or $f_{21}$, but those entries are constant in $t$ for all $tf + (1-t)\beta(f)$.) So we have shown $\alpha\circ \beta \simeq \id_{H^\circ}$.

  For the other composition, we factor the linear maps $\beta(f)$ and $f$ as follows.  From the definition, $\beta$ factors as $\beta(f) = \nu\circ \hat{f}$.  Let $\phi\colon \C^\infty\otimes \C^\infty \to \C^\infty$ be defined by $\phi(\ee_i\otimes \ee_j) = \ee_i$, so $f = \phi\circ \hat{f}$.  The two factorizations are summarized by the diagram
  \[
  \begin{tikzcd}
  \C^\infty  \ar[rr,shift left, "\beta(f)"] \ar[rr,shift right, "f" below]\ar[rd,"\hat{f}" below] & & \C^\infty \\
   & \C^\infty\otimes \C^\infty \ar[ru,shift left, "\nu"] \ar[ru, shift right, "\phi" below].
  \end{tikzcd}
  \]
  The homotopy between $f$ and $\beta(f)$ factors as
  \[
    tf + (1-t)\beta(f) = (t\phi + (1-t)\nu)\circ\hat{f}.
  \]
  To show $\beta\circ\alpha \simeq \id_{H^{\circ\circ}}$, we must show $tf + (1-t)\beta(f)$ is injective whenever $f$ is; for this, it suffices to show that $t\phi + (1-t)\nu$ is injective for all $t\neq 1$.  We do this by direct computation.

  Given $v = \sum_{i,j} a_{ij} \ee_i\otimes\ee_j$, suppose $(t\phi + (1-t)\nu)(v)=0$ for some $t\neq 1$.  Expanding the coefficients, we obtain equations
  \begin{align*}
    a_{11} &= -t(a_{12} + a_{13} + \cdots ), \\
    a_{21} &= -t(a_{22} + a_{23} + \cdots ),
    \intertext{and}
    a_{ij} &=  \frac{t}{t-1}\sum_{\ell\geq1} a_{\nu(i,j),\ell}
  \end{align*}
  for all other $(i,j)$.  Since $\nu(i,j)>i$ for all $(i,j)$ not equal to $(1,1)$ or $(2,1)$, these linear equations express $a_{ij}$ as a combination of $a_{k,\ell}$ for pairs $(k,\ell)$ greater than $(i,j)$ in lexicographic order.  Since all but finitely many coefficients $a_{k,\ell}$ are $0$, this implies $v=0$, completing the proof of the lemma.
\end{proof}

\bigskip
\noindent
{\it Acknowledgements.}  DA thanks Bill Fulton, Mehmet Onat, and Liz Vivas for helpful conversations and correspondence, with special thanks to Matthias Wendt for detailed comments on an early draft.  We also thank the referee for reading carefully and providing several helpful suggestions and corrections.



\end{document}